\documentclass[10pt]{article}
\usepackage{t1enc}
\usepackage[latin1]{inputenc}
\usepackage[english]{babel}
\usepackage{amsmath,amsthm}
\usepackage{amsfonts}
\usepackage{latexsym}
\usepackage[dvips]{graphicx}
\usepackage{graphicx}
\usepackage[natural]{xcolor}

\usepackage{listings}
\title{List strong edge-coloring of graphs with maximum degree 4}

\author{Baochen Zhang$^{1}$, Yulin Chang$^{1}$, Jie Hu$^{1}$, Meijie Ma$^{2}$, Donglei Yang$^{1}$\thanks{Corresponding author. Email:dlyang120@163.com, baochen\_zhang@163.com.},\\
$^{1}$\small School of Mathematics, Shandong University, Jinan, Shandong 250100, China\\
$^{2}$\small School of Management Science and Engineering, \\
\small Shandong Technology and Business University, Yantai, 264005, China}

\date{}

\newtheorem{theorem}{Theorem}
\newtheorem{lemma}{Lemma}[section]
\newtheorem{claim}{Claim}[section]

\usepackage{indentfirst}
\usepackage{graphicx, epsfig, subfigure}
\usepackage{epstopdf}

\begin{document}
\baselineskip 0.595cm
\maketitle

\begin{abstract}
\normalsize
\vspace{3mm}

A strong edge-coloring of a graph $G$ is an edge-coloring such that any two edges on a path of length three receive distinct colors. We denote the strong chromatic index by $\chi_{s}'(G)$ which is the minimum number of colors that allow a strong edge-coloring of $G$. Erd\H{o}s and Ne\v{s}et\v{r}il conjectured in 1985 that the upper bound of $\chi_{s}'(G)$ is $\frac{5}{4}\Delta^{2}$ when $\Delta$ is even and $\frac{1}{4}(5\Delta^{2}-2\Delta +1)$ when $\Delta$ is odd, where $\Delta$ is the maximum degree of $G$. The conjecture is proved right when $\Delta\leq3$. The best known upper bound for $\Delta=4$ is 22 due to Cranston previously. In this paper we extend the result of Cranston to list strong edge-coloring, that is to say, we prove that when $\Delta=4$ the upper bound of list strong chromatic index is 22.

\end{abstract}

\section{Introduction}

A \emph{strong edge-coloring} is a proper edge-coloring with the further condition that no two edges with the same color on a path of length three. To be more clearly, a \emph{strong $k$-edge-coloring} of a graph $G$ is a coloring $\phi: E(G)\longrightarrow [k]$ such that if any two edges $e_1$ and $e_2$ are either adjacent to each other or adjacent to a common edge, then $\phi(e_1)\neq \phi(e_2)$. The \emph{strong chromatic index} of $G$, denoted by $\chi'_{s}(G)$, is the minimum positive integer $k$ for which $G$ has a strong $k$-edge-coloring.

A \emph{list strong edge-coloring} of $G$ is a strong edge-coloring such that each edge $e$ receives a color in a prescribed color list $L(e)$. Let \emph{list assignment} $L=\{L(e): e\in E(G)\}$.  Then graph $G$ is \emph{strongly $L$-edge-colorable} if there exists a strong edge-coloring $c$ of $G$ such that $c(e)\in L(e)$ for every $e\in E(G)$. For a positive integer $k$, a graph $G$ is \emph{strongly $k$-edge-choosable} if $G$ is strongly $L$-edge-colorable for every $L$ with $|L(e)|\geq k$ for all $e\in E(G)$. The \emph{strong choice number}, denoted by $\chi'_{ls}(G)$, is the minimum positive integer $k$ for which $G$ is strongly $k$-edge-choosable.

We consider $\chi_{s}'(G)$ of graphs with known maximum degree and denote the maximum degree of a graph by $\Delta$. As for the strong chromatic number $\chi_{s}'(G)$, Erd\H{o}s and Ne\v{s}et\v{r}il \cite{erdos1,erdos2} conjectured that $\chi_{s}'(G)$ of a graph $G$ is at most $\frac{5}{4}\Delta^{2}$ when $\Delta$ is even and $\frac{1}{4}(5\Delta^{2}-2\Delta +1)$ when $\Delta$ is odd in 1985; they also give a construction to show that if the conjecture is true, then the bound is tight. For graphs with $\Delta=3$, the conjecture was proved right by Andersen \cite{andersen} and by Hor\'{a}k \cite{horak} independently. For $\Delta=4$, while the conjecture says that $\chi_{s}'(G)\leq20$, the best known upper bound is 22 due to Cranston \cite{cranston}.

When $\Delta$ is sufficiently large, Bonamy, Perrett, and Postle~\cite{BPP} proved that $\chi'_{s}(G)\leq 1.835\Delta^2$. As for $k$-degenerate graphs, Yu~\cite{Yu} has proved that $\chi'_{s}(G)\leq (4k-2)\Delta-2k^2+k+1$. More results of this kind can be found in~\cite{Wang,Yang}.

In this paper, we mainly prove the following theorem which extends Cranston's result to the list version.

\begin{theorem}\label{th1}
Let $G$ be a graph with maximum degree 4, then $G$ is strongly 22-edge-choosable.

\end{theorem}

We will give the proof of Theorem~\ref{th1} in section 4. Section 2 introduces definitions and tools that are used in this paper. In section 3, we explore some basic properties of the minimal counterexample to Theorem~\ref{th1}.


\section{Preliminaries and Notation}

Throughout this paper, when we use term coloring, we mean list strong edge-coloring. Each connected component of a graph $G$ can be colored independently, so we assume that $G$ is connected, and we allow our graphs to include loops and multiple edges.

We use $\delta$ to denote the minimum degree of a graph and $d(v)$ to denote the degree of a vertex $v$. The length of the shortest cycle in a graph $G$ is denoted by girth $g(G)$. We define \emph{$i$-cycle} to be a cycle of length $i$.

We denote the distance between two edges by the minimum distances between their endpoints. The \emph{neighborhood} $N(e)$ of an edge $e$ is the set of edges that have distances at most one away from $e$. Intuitively, this is the edge set whose colors could potentially restrict the colors of $e$.

A \emph{partial coloring} of $G$ is a coloring of a proper subgraph of $G$. Given a partial coloring $c$ and an edge $e$, a color in $L(e)$ is \emph{available} for $e$ if the color is not used in $N(e)$. And we denote the set of available colors for $e$ by $L'(e)$. Let $N'(e)$ be the set of colored edges in $Neigh(e)$.

Fig.1 shows that $|N(e)|\leq24$ for the edge $e$. In all figures of this paper, black vertices have no other neighbors than those represented, but white vertices might have other neighbors.

\begin{figure}[h]
  \centering
  \includegraphics[scale=0.45]{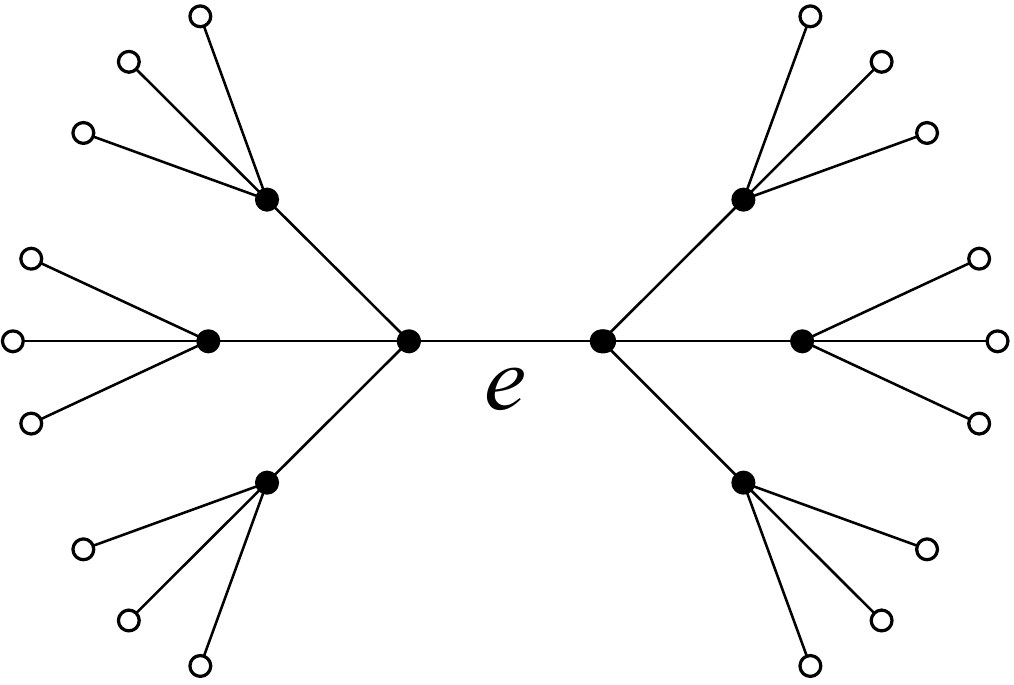}\\
  {Fig.1. The largest possible neighborhood of an edge.}
\end{figure}

One of the main tools we use is the Combinatorial Nullstellensatz.

\begin{lemma}\label{lemma1} \emph{(Alon \cite{alon}, Combinatorial Nullstellensatz)} Let $\mathbb{F}$ be an arbitrary field, and let
$P = P(x_1, \cdots, x_n)$ be a polynomial in $\mathbb{F}[x_1, \cdots, x_n]$. Suppose the degree $deg(P)$ of $P$ equals $\sum\limits^n_{i=1}k_i$, where each $k_i$ is a non-negative integer, and suppose the coefficient of $x_1^{k_1}\cdots x_n^{k_n}$
in $P$
is non-zero. Then if $S_1,\cdots, S_n$ are subsets of $\mathbb{F}$ with $|S_i| > k_i, i=1,\cdots,n$, there exist $s_1\in S_1,\cdots, s_n\in S_n$
so that $P(s_1, \cdots, s_n)\neq0$.
\end{lemma}

Another tool we use is the Hall's Theorem.

\begin{lemma}\label{lemma2} \emph{(Hall \cite{hall})} Let $A_{1},...,A_{n}$ be $n$ subsets of a set U. A distinct representatives of $\{A_{1},...,A_{n}\}$ exists if and only if for all $k$, $1\leq k\leq n$ and every choice of subcollection of size $k$, $\{A_{i_{1}},...,A_{i_{k}}\}$, we have $|A_{i_{1}}\bigcup ...\bigcup A_{i_{k}}|\geq k$.
\end{lemma}

Suppose that we have a partial coloring of $G$, with only the edge set $T$ left uncolored. Let $L'(e)$ be the available color list for each $e\in T$. Lemma~\ref{lemma2} guarantees that if we are unable to complete the coloring by giving each edge its own color, then there exists a set $S\subseteq T$ with $|S|>|\bigcup\limits_{e\in S}L'(e)|$. Define the \emph{discrepancy}, $disc(S)=|S|-|\bigcup\limits_{e\in S}L'(e)|$. The following lemma is a generalization to the list version of Lemma 1 in~\cite{cranston}.

\begin{lemma}\label{lemma3} Let $T$ be the set of uncolored edges in a partially colored graph. Let $S$ be a subset of $T$ with maximum discrepancy. Then any coloring of $S$ can be extended to a coloring for $T$.
\end{lemma}

\begin{proof}Assume the claim is false. Since the coloring of $S$ cannot be extended to $T\backslash S$, some set of edges $S'\subseteq T\backslash S$ has positive discrepancy (after coloring $S$). We show that $disc(S\cup S')>disc(S)$. Let $R=\bigcup\limits_{e\in S\cup S'}L'(e)$, $R_1=\bigcup\limits_{e\in S}L'(e)$, $k=disc(S)$. Let $R_2=\bigcup\limits_{e\in S'}L'(e)$ after the edges of $S$ have been colored. Then $|S|=k+|R_1|$ and $|S'|\geq1+|R_2|$. Since $S$ and $S'$ are disjoint, we get
$$|S\cup S'|=|S|+|S'|\geq k+1+|R_1|+|R_2|>k+|R|.$$

The latter inequality holds since a color in $R\backslash R_1$ must be in $R_2$ and therefore we have $|R|=|R_1\cup R_2|\leq|R_1|+|R_2|$. Hence
$disc(S\cup S')=|S\cup S'|-|R|>k=disc(S)$. This contradicts the maximality of $disc(S)$. Thus, any coloring of $S$ can be extended to a coloring of $T$.
\end{proof}

In the last part of this section, we prove a useful lemma for our main proof.

Let $v$ be an arbitrary vertex of a graph $G$. Let $dist(v,v_1)$ denote the distance from vertex $v_{1}$ to $v$. And the distance from any edge $e$ to $v$ is denoted by $dist_{v}(e)=\min\limits_{u\in e}dist(v,u)$. Let distance class $i$ be the set of edges that are at distance $i$ from $v$. We call an edge ordering is \emph{compatible} with vertex $v$ if $e_{1}$ precedes $e_{2}$ in the ordering only when $dist_{v}(e_{1})\geq dist_{v}(e_{2})$. Similarly, if we specify a cycle $C$ in the graph, let $dist_{C}(e)=\min\limits_{u\in e,v\in C}dist(v,u)$ denote the distance from edge $e$ to $C$, so we can define distance class $i$ to be the set of edges that are at distance $i$ from $C$ and call an edge ordering is \emph{compatible} with $C$ if $e_{1}$ precedes $e_{2}$ in the ordering only when $dist_{C}(e_{1}) \geq dist_{C}(e_{2})$.

\begin{lemma}\label{lemma4} Let $G$ be a graph with maximum degree 4, $v$ is an arbitrary vertex of $G$, then $G-v$ is strongly 21-edge-choosable. If $C$ is a cycle of length at least 3 in $G$, then $G-E(C)$ is strongly 21-edge-choosable.
\end{lemma}

\begin{proof}We consider the case when only the edges incident with the single vertex $v$ are left uncolored. Color the other edges in an ordering that is compatible with $v$. During this process, suppose that we are coloring edge $e$ which is not incident with $v$, let $u$ be a vertex adjacent to an endpoint of $e$ that is on a shortest path from $e$ to $v$. Then none of the four edges incident with $u$ has been colored, since each edge incident with $u$ is in a lower distance class than $e$ (Fig.2). Thus, $|N'(e)|\leq24-4=20$, then we can always find a color available for $e$.

\begin{figure}[h]
  \centering
  \includegraphics[scale=0.45]{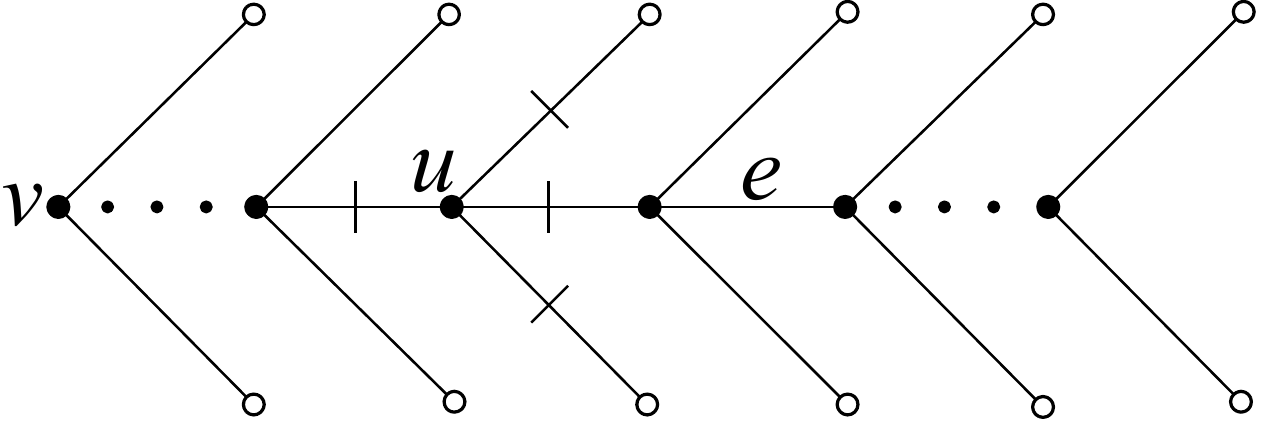}\\
  {Fig.2. The relationship between $u$ and $e$.}
\end{figure}

To prove the case when only the edges of $C$ are left uncolored, we color the other edges in an ordering compatible with $C$. The argument above holds for every edge that is not incident with $C$. If $e$ is incident with $C$ and $|C|\geq4$, then at least four edges in $N(e)$ are edges of $C$; so again $|N'(e)|\leq24-4=20$. If $e$ is incident with $C$ and $|C|=3$, then $|N(e)|\leq23$. The three uncolored edges of $C$ imply that $|N'(e)|\leq23-3=20$.
\end{proof}

Lemma~\ref{lemma4} shows that if $\Delta=4$, we can color nearly all edges with color list of length 21. In the rest of this paper, we show that we can always complete the coloring with color list of length 22.
%
%

\section{Basic Properties}

Let $G$ be a minimal counterexample to Theorem~\ref{th1}, which means that if there is a list assignment $L$, then $G$ is not strongly $L$-edge-colorable but any proper subgraph of $G$ is strongly $L$-edge-colorable. In this section, we show that $G$ is a simple 4-regular graph and $g(G)\geq6$.


\begin{lemma}\label{lemma5} $G$ is 4-regular.
\end{lemma}

\begin{proof}Suppose $G$ is not 4-regular. Let $v$ be a vertex of $G$ with $d(v)=3$ (the case $d(v)<3$ is easier to prove). Color the edges in an ordering that is compatible with $v$. Let $e_{1}$, $e_{2}$, $e_{3}$ be the edges that are incident with $v$. If the edges are ordered $e_{1}$, $e_{2}$, $e_{3}$, we have $|N'(e_{1})|\leq18$, $|N'(e_{2})|\leq19$ and $|N'(e_{3})|\leq20$, which means $|L'(e_{1})|\geq4$, $|L'(e_{2})|\geq3$ and $|L'(e_{3})|\geq2$, so there are enough colors for $e_{1}$, $e_{2}$ and $e_{3}$.
\end{proof}

\begin{lemma}\label{lemma6} $G$ is simple.
\end{lemma}

\begin{proof}Suppose $G$ is not simple. If $G$ has a loop $e_1$ incident with a vertex $v$, let $e_{2}$, $e_{3}$ be the edges incident with $v$ which are not loops. Then color the edges in an ordering that is compatible with $v$. We have $|N'(e_{1})|\leq8$, $|N'(e_{2})|\leq16$, $|N'(e_{3})|\leq15$, thus there are many colors available for $e_{1}$, $e_{2}$ and $e_{3}$. Next we consider the other case when $G$ has multiple edges.

Let $v$ be a vertex in a 2-cycle. Color the edges in an ordering that is compatible with vertex $v$, let $e_{3}$, $e_{4}$ belong to the 2-cycle and $e_{1}$, $e_{2}$ be the other edges incident with $v$. Then $|N'(e_{1})|\leq17$, $|N'(e_{2})|\leq18$, $|N'(e_{3})|\leq16$ and $|N'(e_{4})|\leq17$, so there are available colors for $e_{1}$, $e_{2}$, $e_{3}$ and $e_{4}$.
\end{proof}

\begin{lemma}\label{lemma7} $G$ has no 3-cycle.
\end{lemma}

\begin{proof}Suppose $G$ has a 3-cycle $C$. By Lemma~\ref{lemma4} we color all edges except the edges of $C$. We observe that $|N(e)|\leq20$ for every edge $e$ in $C$, so $|L'(e)|\geq4$ and we can finish the coloring.
\end{proof}

\begin{lemma}\label{lemma8} $G$ has no 4-cycle.
\end{lemma}

\begin{proof}Suppose $G$ has a 4-cycle $C$, with all edges labeled in Fig.3 (a). We denote $a_{i}$ and $b_{i}$ by \emph{pendant edges}. If two pendant edges share an endpoint not on $C$, then the two edges form an \emph{adjacent pair}. The only possibility of an adjacent pair is that $a_1$ or $b_1$ shares an endpoint with $a_3$ or $b_3$ (or similarly $a_2$ or $b_2$ shares an endpoint with $a_4$ or $b_4$). So we call ($a_1$, $b_1$, $a_3$, $b_3$) a \emph{pack} and ($a_2$, $b_2$, $a_4$, $b_4$) is also a pack. By Lemma~\ref{lemma4}, we can color all edges except the edges shown in Fig.3 (a). We will prove this lemma by considering the number of adjacent pairs.

\noindent\textbf{Case 1.} If there are at least two adjacent pairs, then we color the pendant edges by Lemma~\ref{lemma4}. We have $|N(c_i)|\leq21$ and so $|L'(c_i)|\geq4$ for each $i\in[4]$, thus we can color the four edges on $C$.

Next we will discuss the cases that we have exactly one adjacent pair and we have no adjacent pairs. These two cases both have at leat one pack of edges that share their endpoints only on $C$ (otherwise it belongs to Case 1). But it is possible that a pair of nonadjacent edges in a pack has an edge adjacent to them. We call the edge between the pair as \emph{diagonal edge}, and we can easily observe that there are at most 4 diagonal edges of a pack (see Fig.3 (b)). Before discuss the last two cases, we first give a claim about diagonal edges.

\begin{figure}[h]
  \centering
  \includegraphics[scale=0.75]{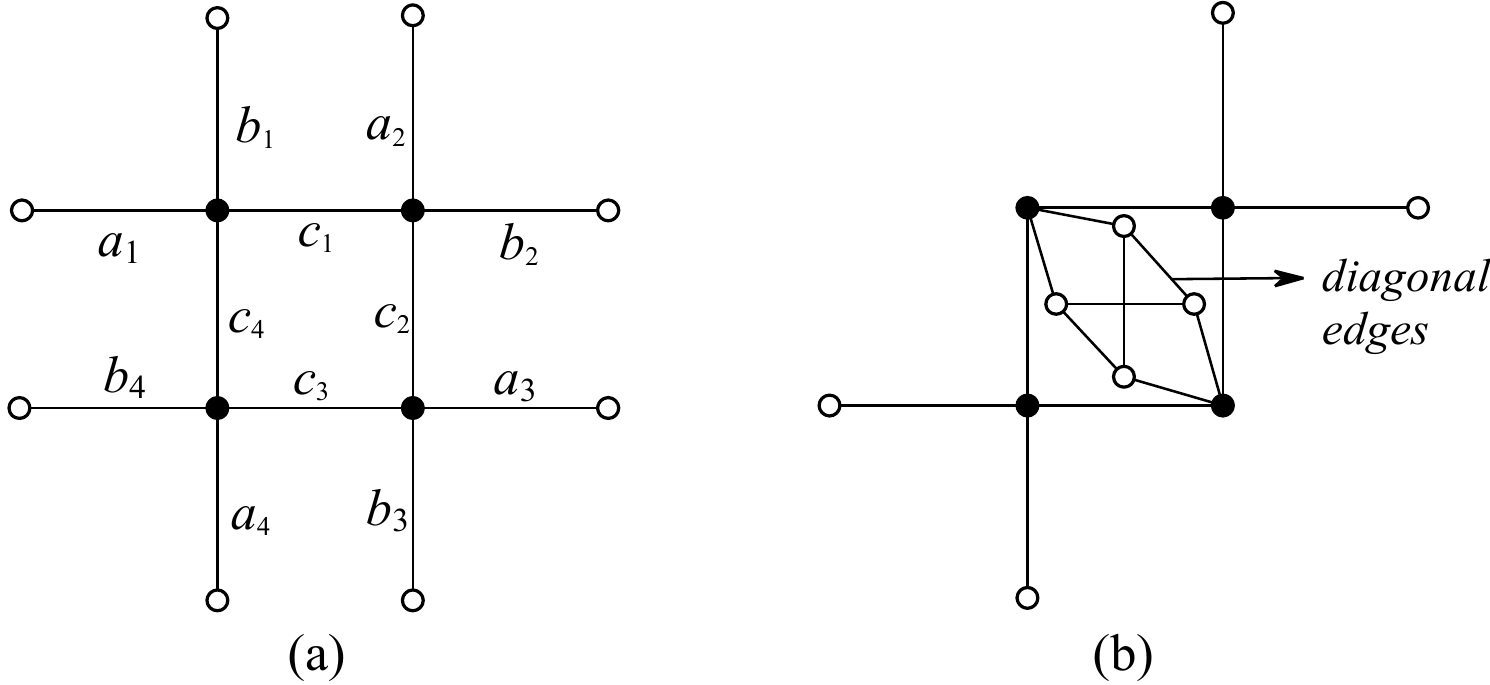}\\
  {Fig.3. (a) A 4-cycle in $G$. (b) Four diagonal edges of a pack.}
\end{figure}
\begin{claim}\label{claim1}
If there exist 4 diagonal edges of a pack, then $G$ is strongly 22-edge-choosable.
\end{claim}
\begin{proof}Observe that the neighborhood of a diagonal edge has size at most 21. Thus we color all edges except the four edges of $C$ and the four diagonal edges by Lemma~\ref{lemma4}. Now we color the four edges of $C$ (the four uncolored diagonal edges ensure there are enough colors available for edges of $C$). Lastly, we color the four diagonal edges.
\end{proof}

So according to Claim~\ref{claim1}, we only need to consider the case that there is at most 3 diagonal edges of a pack in the following discussion, that is, at least one pair of nonadjacent edges of a pack do not have diagonal edges.

\noindent\textbf{Case 2.} Suppose the uncolored edges contain exactly one adjacent pair. Without loss of generality, suppose edges $a_2$ and $a_4$ share an endpoint. So we have edges $a_1$, $b_1$, $a_3$ and $b_3$ as a pack and at least one pair of nonadjacent edges of this pack do not have diagonal edges. We suppose that $a_1$ and $a_3$ is such a pair.

Observe that $|L'(c_i)|\geq11$ for each $i\in[4]$ and $|L'(a_i)|=|L'(b_i)|\geq7$ for $i=1,3$. If there is a color $x\in \bigcup\limits_{i=1}^4(L'(a_i)\bigcup L'(b_i))$, but $x\notin L(c_j)$ for some $j\in[4]$ (suppose that color $x\in L'(a_1)\setminus L'(c_1)$), then we can give $a_1$ color $x$ and color the uncolored pendant edges by Lemma~\ref{lemma4}. As for the four edges of cycle $C$, we have $|L'(c_i)|\geq3$ for $i=2,3,4$ and $|L'(c_1)|\geq4$, so we can color the edges of the 4-cycle in the order $c_4$, $c_3$, $c_2$, $c_1$. Then if one of the $|L'(c_i)|>11$ (suppose that is $L'(c_1)$), we can also color the whole graph with similar coloring strategy above because after coloring pendant edges by Lemma~\ref{lemma4} we can have $|L'(c_i)|\geq3$ for $i=2,3,4$ and $|L'(c_1)|\geq4$ again.

Thus the last situation remaining is that $\bigcup\limits_{i=1}^4(L'(a_i)\bigcup L'(b_i))\subseteq L'(c_j)$ and $|L'(c_j)|=11$, for each $j\in[4]$. Then we have $L'(a_1)\bigcup L'(a_3)\subseteq L'(c_1)$ and $|L'(c_1)|=11$, so $|L'(a_1)\bigcup L'(a_3)|\leq11$, but $|L'(a_1)|\geq7$, $|L'(a_3)|\geq7$. That means there are at least 3 in $L'(a_1)\bigcap L'(a_3)$, so we can give $a_1$ and $a_3$ color $y\in L'(a_1)\bigcap L'(a_3)$, then color the uncolored pendant edge by Lemma~\ref{lemma4} again. It follows that $|L'(c_i)|\geq4$ for each $i\in[4]$.

\noindent\textbf{Case 3.} Finally, suppose that the uncolored edges contain no adjacent pairs. In this case we will use Lemma~\ref{lemma2} to simplify our proof. If we cannot assign a distinct color to each uncolored edge, then Lemma~\ref{lemma2} guarantees there exists a subset of the 12 uncolored edges with positive discrepancy. Let $S$ be a subset of the uncolored edges with maximum discrepancy. We observe that if $e$ is an edge of $C$, then $|L'(e)|\geq10$ and if $e$ is an pendant edge then $|L'(e)|\geq7$. We will assume that $S$ contains some edge of $C$, otherwise we can color $S$ by Lemma~\ref{lemma4}, then extend the coloring to the remaining uncolored edges by Lemma~\ref{lemma3}. Since $disc(S)>0$ and $|L'(e)|\geq10$ for each edge $e$ of $C$, $|S|$ is 11 or 12.

Suppose $|S|=12$, so two packs $(a_1$, $b_1$, $a_3$, $b_3)$ and $(a_2$, $b_2$, $a_4$, $b_4)$ are all in $S$. According to Claim~\ref{claim1}, each pack has a nonadjacent pair without a diagonal edge, we suppose that the two pairs are $(a_1$, $a_3)$ and $(a_2$, $b_4)$. Note that $|L'(a_1)|\geq7$, $|L'(a_3)|\geq7$, and $|L'(a_1)\bigcup L'(a_3)|\leq|\bigcup\limits_{e\in S}L'(e)|=|S|-disc(S)\leq11$, then $|L'(a_1)\bigcap L'(a_3)|\geq3$. Similarly, we can get $|L'(a_2)\bigcap L'(b_4)|\geq3$. So we can choose a color $x\in L'(a_1)\bigcap L'(a_3)$ to color $a_1$ and $a_3$. Because $|L'(a_2)\bigcap L'(b_4)|\geq3$, we can choose another color $y$ to color $a_2$ and $b_4$. After coloring the remaining uncolored pendant edges by Lemma~\ref{lemma4}, we have $|L'(c_i)|\geq4$ for each $i\in[4]$.

Suppose $|S|=11$. Then there is an uncolored edge $e\in S$. If $e\in C$, then we can use the same strategy as $|S|=12$ and get $|L'(c_i)|\geq4$ again. So discuss the case that $e$ is an pendant edge, suppose that $e$ is $a_1$. Since $(a_2$, $b_2$, $a_4$, $b_4)$ is still a pack, there is a nonadjacent pair without diagonal edge by Claim~\ref{claim1} (suppose that is $(a_2$, $a_4)$). Then again we can give $a_2$, $a_4$ the same color and color the remaining uncolored pendant edges of $S$ by Lemma~\ref{lemma4}. Since $|N(c_i)|\leq23$ for each $i\in[4]$, $a_1$ is uncolored and on the  neighborhood of each $c_i$, $a_2$, $a_4$ have the same color and are also in their neighborhood, we have $|L'(c_i)|\geq4$ again.

As we can give $S$ a coloring, we can have a coloring for the 12 uncolored edges by Lemma~\ref{lemma3}.
\end{proof}

\begin{lemma}\label{lemma9} $G$ has no 5-cycle.
\end{lemma}
\begin{proof}As shown in Fig.4 (a), suppose $G$ has a 5-cycle $C$, with similar label strategy as in Lemma~\ref{lemma8}. We also refer to the edges labeled by $a_i$ and $b_i$ as pendant edges. We claim that at least one of $a_4$ and $b_4$ is not in the neighborhood of $b_2$; for otherwise, we have a 4-cycle. Thus we can assume that there is no edge between $b_1$ and $b_3$. Similarly, we assume that there is no edge between the following pairs: $(b_2, b_5)$, $(b_5, b_3)$. We color all edges except five edges on $C$ by Lemma~\ref{lemma4} and then erase the colors of $b_2$, $b_3$, $b_4$, $b_5$. Note now we have $|L'(c_i)|\geq5$ for $i=1,2,4,5$, $|L'(c_3)|\geq6$, $|L'(b_2)|\geq3$, $|L'(b_5)|\geq3$ and $|L'(b_3)|\geq4$, $|L'(b_4)|\geq4$. We relabel the edges as shown in Fig.4 (b). Then for every $x_i$, $S_i=L'(x_i)$. We will color $x_i$ with color $s_i\in S_i$ ($i\in[9]$) respectively. If $s_i-s_j\neq0$ for any two edges $x_i$, $x_j$ ($i\in[9]$) that are possibly adjacent or lie on a path of length 3, then we have a strong edge-coloring. Hence we will finish the coloring if there exist color $s_i\in S_i$ for each $i\in[9]$ such that polynomial $P(s_1, \cdots, s_n)\neq0$ where
\begin{align*}
P(x_1,x_2,x_3,x_4,x_5,x_6,x_7,x_8,x_9)=&\prod\limits_{2\leq i\leq7}(x_1-x_i)\prod\limits_{3\leq j\leq8}(x_2-x_j)\prod\limits_{4\leq k\leq9}(x_3-x_k)\\
&(x_1-x_9)(x_4-x_5)(x_4-x_7)(x_4-x_8)(x_4-x_9)\\
&(x_5-x_6)(x_5-x_8)(x_5-x_9)(x_6-x_7)(x_7-x_8)\\
&(x_8-x_9).
\end{align*}

\begin{figure}[h]
\begin{center}
  \includegraphics[scale=0.7]{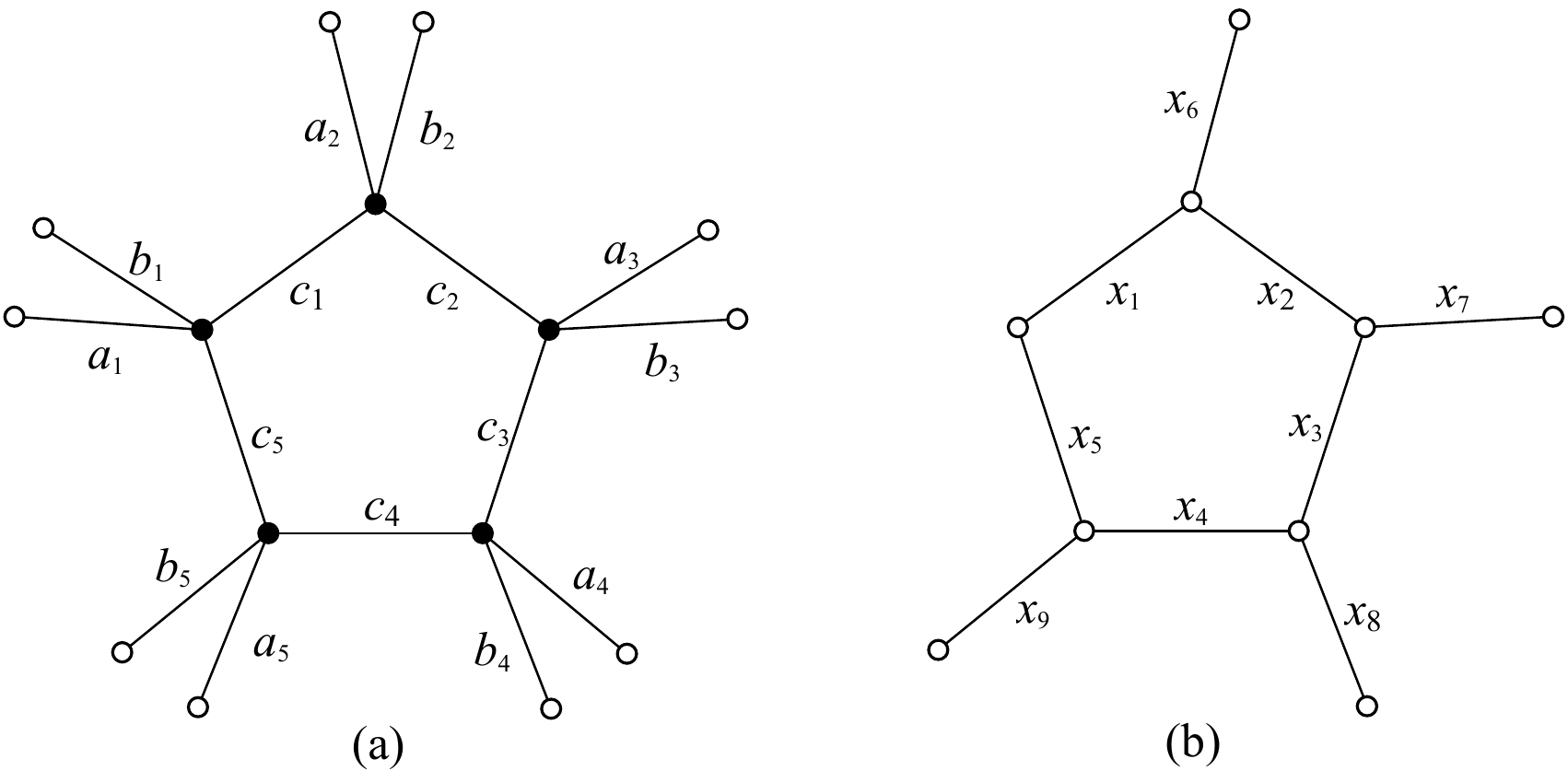}\\
\end{center}
{Fig.4. (a) A 5-cycle in $G$ with pendant edges. (b) The local structure that uses Lemma~\ref{lemma1}.}
\end{figure}

We use MATLAB to calculate the coefficients of specific monomials. The codes are listed in the last section. By MATLAB, we obtain a coefficient $c_P(x_1^3x_2^4x_3^5x_4^4x_5^4x_6^2x_7^3x_8^2x_9^2)=-1\neq0$. According to Lemma~\ref{lemma1}, since $\deg(P)=29=\sum\limits^n_{i=1}k_i$ and $|S_i|> k_i$ ($i\in[9]$), there exist $s_i\in S_i$ ($i\in[9]$) such that $P(s_1,\cdots,s_9)\neq0$. Coloring $x_1,\cdots,x_9$ with $s_1,\cdots,s_9$ respectively and then we obtain a coloring which makes $G$ strongly 22-edge-choosable.
\end{proof}

\section{Proof of Theorem~\ref{th1}}
In this section, we give a proof of Theorem~\ref{th1}. Note that $G$ is a minimal counterexample to Theorem~\ref{th1}. By Lemma~\ref{lemma5}- \ref{lemma9}, we know that $G$ is a simple 4-regular graph with $g(G)\geq6$.

\begin{proof}[Proof of Theorem~\ref{th1}]Let $v$ be an arbitrary vertex in $G$, with the 4 incident edges labeled $e_i$ ($i\in[4]$) in clockwise order and denote $A_i$ ($i\in[4]$) by the edge set which contains the three edges adjacent to $e_i$ but not incident with $v$ (as shown in Fig.5). By Lemma~\ref{lemma4}, we know that $G-v$ is strongly 21-edge-choosable. Now we prove the following claim which follows from the similar argument as in the proof of Lemma~\ref{lemma4}.

\begin{figure}[h]
  \centering
  \includegraphics[scale=0.7]{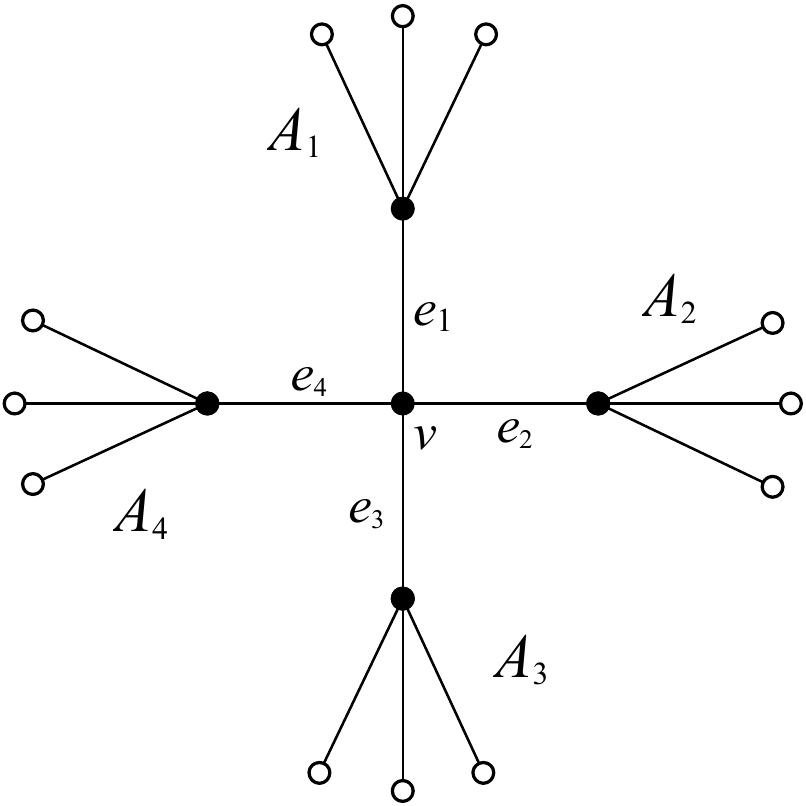}\\
{Fig.5. Vertex $v$ has degree 4 and $g(G)\geq6$.}
\end{figure}

\begin{claim}\label{claim2}
If we precolor one edge from each $A_i$, then $G-v$ is strongly 22-edge-choosable.
\end{claim}

\begin{proof} We adapt idea of Lemma~\ref{lemma4} to show that in the presence of four specific precolored edges $G-v$ is strongly 22-edge-choosable. Lemma~\ref{lemma4} argued that there are at least four uncolored edges in the neighborhood of the edge being coloring, so $|N'(e)|\leq20$. The same argument applies in this case except that possibly one of the edges that was uncolored in Lemma~\ref{lemma4} is precolored. Hence $|N'(e)|\leq21$ (this follows from the fact that at most one of the four uncolored edges in Lemma~\ref{lemma4} that are incident with the same vertex $u$ is precolored). Thus, $G-v$ is strongly 22-edge-choosable.
\end{proof}

So according to Claim~\ref{claim2}, we can precolor four edges from different $A_i$ ($i\in[4]$). Let $A_i=\{a_i,b_i,c_i\}$ and denote the color list of $A_i$ by $L(A_i)=L(a_i)\bigcup L(b_i)\bigcup L(c_i)$ for each $i\in[4]$. Next we consider the relationship between $L(A_i)$ ($i\in[4]$). For simplicity, in the following proof, if a color belongs to $L(A_i)$, without loss of generality, we assume it belongs to $L(a_i)$.

\noindent\textbf{Case 1.} If $L(A_i)\bigcap L(A_j)=\phi$ ($i\in[4]$), then we have $|\bigcup\limits_{i=1}^{4}L(A_i)|\geq4\times22>3\times22\geq |\bigcup\limits_{i=1}^{3}L(e_i)|$. So we can get a color $x_1\in\bigcup\limits_{i=1}^{4}L(A_i)\setminus\bigcup\limits_{i=1}^{3}L(e_i)$, then we assume $x_1\in L(a_1)$. Similarly, we can get color $x_2\in \bigcup\limits_{i=2}^{4}L(A_i)\setminus\bigcup\limits_{i=1}^{2}L(e_i)$ (suppose $x_2\in L(a_2)$) and $x_3\in L(A_3)\bigcup L(A_4)\setminus L(e_1)$ (suppose $x_3\in L(a_3)$). Then we can precolor $a_1$, $a_2$, $a_3$ with $x_1$, $x_2$, $x_3$ respectively, and color all uncolored edges except incident edges of $v$ according to Claim~\ref{claim2}. Then we have $|L'(e_1)|\geq4$, $|L'(e_2)|\geq3$, $|L'(e_3)|\geq2$ and $|L'(e_4)|\geq1$, so we can color the four edges in the order $e_4$, $e_3$, $e_2$, $e_1$.

\noindent\textbf{Case 2.} If $\bigcap\limits_{i=1}^{4}L(A_i)\neq\phi$, we assume $x\in\bigcap\limits_{i=1}^{4}L(A_i)$. Then we give color $x$ to $a_i$ for each $i\in[4]$ (the four edges can receive the same color since $g(G)\geq6$) and color all uncolored edges except the four incident edges of $v$. We observe that each $e_i$ satisfies $|L(e_i)|\geq4$.

\noindent\textbf{Case 3.} If there exists a common color in the list of some three edge sets but not of four edge sets $A_i$, suppose $x\in \bigcap\limits_{i=1}^{3}L(A_i)\setminus L(A_4)$, then we can color $a_i$ $(1\in[3])$ with $x$. And we also get that $|L(A_1)\bigcup L(A_4)|\geq22+1=|L(e_1)|+1$, so there is at least a color $y$ such that $y\in L(A_1)\bigcup L(A_4)\setminus L(e_1)$.

If we can find a color $y\in L(A_4)\setminus L(e_1)$, then we give color $y$ to $a_4$ and color the remaining edges of $G-v$ by Claim~\ref{claim2}. So we have $|L'(e_1)|\geq4$, $|L'(e_i)|\geq3$ $(2\leq i\leq4)$ and color the four edges in the order $e_4$, $e_3$, $e_2$, $e_1$. If we cannot find such $y$, then we know that $L(A_4)=L(e_1)$. Since $x\notin L(A_4)$, $x\notin L(e_1)$. We color the remaining edges of $G-v$ by Claim~\ref{claim2} and get $|L'(e_1)|\geq4$, $|L'(e_i)|\geq3$ $(2\leq i\leq4)$ again.

\noindent\textbf{Case 4.} If there exists a common color in the list of some two edge sets but not of three edge sets, suppose that $L(A_1)\bigcap L(A_2)\neq\phi$ but $L(A_1)\bigcap L(A_2)\bigcap L(A_i)=\phi$ ($i=3$ or 4). Suppose $x\in L(A_1)\bigcap L(A_2)$. Before discussing this case, we provide a claim.

\begin{claim}\label{claim3}
In Case 4, if there exists a color $x\in L(A_i)\bigcap L(A_j)$ $(1\leq i<j\leq4)$, then $x\in \bigcap\limits_{k=1}^{4}L(e_k)$.
\end{claim}

\begin{proof} Assume the claim is false. Suppose that $x\in L(A_1)\bigcap L(A_2)\setminus L(e_1)$, first give $x$ to $a_1$ and $a_2$. Since $x\notin L(A_3)\bigcup L(A_4)$ and $|L(A_1)\bigcup L(A_3)|\geq22+1=|L(e_i)|+1$, there must be a color $y\in L(A_1)\bigcup L(A_3)\setminus L(e_i)$  for each $(i\in[4])$.

Similar to the last part of Case 3, if we can find a color $y\in L(A_3)\setminus L(e_2)$, then we give $y$ to $a_3$. Then similarly, we have a color $z\in L(A_2)\bigcup L(A_4)\setminus L(e_2)$. If we can find such $z\in L(A_4)$, then give $z$ to $a_4$ and color all uncolored edges except incident edges of $v$ by Claim~\ref{claim2}. So we have $|L'(e_1)|\geq3$, $|L'(e_2)|\geq4$, $|L'(e_i)|\geq2$ $(i=3, 4)$ and can color the four edges in the order $e_4$, $e_3$, $e_1$, $e_2$. If we cannot find such $z$, then we know that $L(A_4)=L(e_2)$ and so $x\notin L(e_2)$. We color all uncolored edges except incident edges of $v$ by Claim~\ref{claim2} and have $|L'(e_1)|\geq3$, $|L'(e_2)|\geq4$, $|L'(e_i)|\geq2$ $(i=3,4)$ again. Lastly, we color the four edges in order $e_4$, $e_3$, $e_1$, $e_2$.

If we cannot find a color like $y\in L(A_3)\setminus L(e_2)$, then we know that $L(A_3)=L(e_2)$ and $x\notin L(e_2)$. Then similarly, we have a color $z\in L(A_2)\bigcup L(A_4)\setminus L(e_2)$, if we can find such $z\in L(A_4)$, then give color $z$ to $a_4$ and color all uncolored edges except incident edges of $v$ by Claim~\ref{claim2}. So we have $|L'(e_1)|\geq3$, $|L'(e_2)|\geq4$, $|L'(e_i)|\geq2$ $(i=3, 4)$ and color the four edges in the order $e_4$, $e_3$, $e_1$, $e_2$. If we cannot find such $z$, then we know that $L(A_4)=L(e_2)$ and thus $L(A_3)=L(A_4)$, so give the same color in $L(A_3)$ to $a_3$, $a_4$ and color all uncolored edges except incident edges of $v$ by Claim~\ref{claim2}. We will get $|L'(e_1)|\geq4$, $|L'(e_2)|\geq4$, $|L'(e_i)|\geq3$ $(i=3,4)$ and can color the four edges in the order $e_4$, $e_3$, $e_2$, $e_1$.
\end{proof}

Since we assume $L(A_1)\bigcap L(A_2)\neq\phi$ but $L(A_1)\bigcap L(A_2)\bigcap L(A_i)=\phi$ ($i=3$ or 4), according to Claim~\ref{claim3}, we infer that $\sum\limits_{1\leq i<j\leq4}|L(A_i)\bigcap L(A_j)|\leq22$. Then by inclusive-exclusive principle, we have that $|\bigcup\limits_{1\leq i\leq4}L(A_i)|=4\times22-\sum\limits_{1\leq i<j\leq4}|L(A_i)\bigcap L(A_j)|\geq3\times22$ and $|\bigcup\limits_{i=1}^{3}L(e_i)|<3\times22$, so we can find a color $y\in\bigcup\limits_{i=1}^{4}L(A_i)\setminus\bigcup\limits_{i=1}^{3}L(e_i)$ (suppose that $y\in L(A_1)$). Similarly we can get color $z\in\bigcup\limits_{i=2}^{4}L(A_i)\setminus L(e_1)\bigcup L(e_2)$ (suppose that $z\in L(A_2)$). And we also know $|L(A_3)\bigcup L(A_4)|\geq22$.

If $|L(A_3)\bigcup L(A_4)|=22$, then we have $L(A_3)=L(A_4)$ and get $\sum\limits_{1\leq i<j\leq4}|L(A_i)\bigcap L(A_j)|\geq22+1$ which is a contradiction.

If $|L(A_3)\bigcup L(A_4)|>22$, there must be a color $w\in L(A_3)\bigcup L(A_4)\setminus L(e_1)$ (suppose that $w\in L(A_3)$), so we can give color $y$, $z$, $w$ to $a_1$, $a_2$, $a_3$ and color all uncolored edges except incident edges of $v$ by Claim~\ref{claim2}. So we have $|L'(e_1)|\geq4$, $|L'(e_2)|\geq3$, $|L'(e_3)|\geq2$, $|L'(e_4)|\geq1$ and color the four edges in the order $e_4$, $e_3$, $e_2$, $e_1$.

The proof of Theorem~\ref{th1} is completed.
\end{proof}

\section{Acknowledgement}

This work was supported by the National Natural Science Foundation of China (11471193, 11631014), the Foundation for Distinguished Young Scholars of Shandong Province (JQ201501), the Fundamental Research Funds of Shandong University and Independent Innovation Foundation of Shandong University (IFYT14012), the Shandong Provincial Nature Science Foundation of China (No. ZR2017MA018).

\section{Appendix}
\lstset{
  language=matlab,
  basicstyle=\small,
  stepnumber=1,
  numbersep=5pt,
commentstyle=\color{red!50!green!50!blue!50},
showstringspaces=false,
basicstyle=\footnotesize,
xleftmargin=2em,xrightmargin=2em,
  tabsize=4,
}
\begin{lstlisting}
%Matlab
%input
syms x1 x2 x3 x4 x5 x6 x7 x8 x9
%Lemma 2.4
Q=(x1-x2)*(x1-x3)*(x1-x4)*(x1-x5)*(x1-x6)*(x1-x7)*(x1-x9)*(x2-x3)*(x2-x4)*(x2-x5)*
  (x2-x6)*(x2-x7)*(x2-x8)*(x3-x4)*(x3-x5)*(x3-x6)*(x3-x7)*(x3-x8)*(x3-x9)*(x4-x5)*
  (x4-x7)*(x4-x8)*(x4-x9)*(x5-x6)*(x5-x8)*(x5-x9)*(x6-x7)*(x7-x8)*(x8-x9);
C1=diff(diff(diff(diff(diff(diff(diff(diff(diff(Q,x1,3),x2,4),x3,5),x4,4),x5,4)
   ,x6,2),x7,3),x8,2),x9,2)/factorial(3)/factorial(4)/factorial(5)/factorial(4)
   /factorial(4)/factorial(2)/factorial(3)/factorial(2)/factorial(2)
%output
C1=-1
\end{lstlisting}


\begin{thebibliography}{99}

\bibitem{alon} N. Alon, Combinatorial Nullstellensatz, Combinatorics Probability and Computing 8 (1999) 7--29.
%
\vspace {-0.25cm}
%
\bibitem{andersen} L. D. Andersen, The strong chromatic index of a cubic graph is at most 10, Discrete Mathematics 108 (1992) 231--252.
%
\vspace {-0.25cm}
%
\bibitem{BPP} M. Bonamy, T. Perrett, L. Postle, Colouring graphs with sparse neighbourhoods, Bounds and Applications, submitted.
%
\vspace {-0.25cm}
%
\bibitem{cranston} D. Cranston, Strong edge-coloring of graphs with maximum degree 4 using 22 colors, Discrete Mathematics 306 (2006) 2772--2778.
%
\vspace {-0.25cm}
%
\bibitem{dai} T. Dai, G. Wang, D. Yang, G. Yu, Strong choice number of subcubic graphs, submitted.
%
\vspace {-0.25cm}
%
\bibitem{erdos1} P. Erd\H{o}s, Problems and results in combinatorial analysis ang graph theory, Discrete Mathematics 72 (1988) 81--92.
%
\vspace {-0.25cm}
%
\bibitem{erdos2} P. Erd\H{o}s, J. Ne\v{s}et\v{r}il, [Problem], in: G. Hal\'{a}sz, V. T. S\'{o}s (Eds.), Irregularities of Partitions, Springer, Berlin, 1989, 161--349.
%
\vspace {-0.25cm}
%
\bibitem{hall} P. Hall, On representatives of subsets, Journal of the London Mathematical Society 1 (1935) 26--30.
%
\vspace {-0.25cm}
%
\bibitem{horak} P. Hor\'{a}k, H. Qing, W. T. Trotter, Induced matchings in cubic graphs, Journal of Graph Theory 17 (1993) 151--160.
%
\vspace {-0.25cm}
%
\bibitem{Wang} T. Wang, Strong chromatic index of $k$-degenerate graphs, Discrete Mathematics 33 (2014) 17--19.
%
\vspace {-0.25cm}
%
\bibitem{Yang} D. Yang, X. Zhu, Strong chromatic index of sparse graphs, Journal of Graph Theory 83 (2016) 334--339.
%
\vspace {-0.25cm}
%
\bibitem{Yu} G. Yu, Strong edge-colorings for $k$-degenerate graphs, Graphs and Combinatorics 31 (2015), 1815-1818.
%
\vspace {-0.25cm}
%
\end{thebibliography}
\end{document}